\documentclass[12pt]{elsarticle}
\usepackage{amssymb}
\usepackage{a4wide,enumerate,xcolor}
\usepackage{amsmath,comment}
\usepackage{amsthm}

\allowdisplaybreaks

\let\pa\partial

\let\eps\varepsilon

\newcommand{\R}{{\mathbb R}}

\newtheorem{theorem}{Theorem}
\newtheorem{lemma}[theorem]{Lemma}

\begin{document}
	\title{Uniqueness of Weak Solutions to One-Dimensional Doubly Degenerate Cross-Diffusion System\tnoteref{t1}}
	
	\author[a]{Xiuqing Chen\corref{cor}}
	\ead{chenxiuqing@mail.sysu.edu.cn}
	
	\author[a]{Bang Du}
	\ead{dubang@mail2.sysu.edu.cn}
	
	\address[a]{School of Mathematics(Zhuhai), Sun Yat-sen University, Zhuhai 519082, Guangdong Province, China}
	
	\cortext[cor]{Corresponding author.}

	\date{\today}
	\tnotetext[t1]{The authors acknowledge support from the National Natural Science Foundation of China (NSFC),  grant 12471206.}
	
	\begin{abstract}
		The uniqueness of global weak solutions to one-dimensional doubly degenerate cross-diffusion system is shown. The equations model the evolution of feeding bacterial populations in a malnourished environment. The key idea of the proof is applying anti-derivative of the sum of weak solutions to the system.
	\end{abstract}
	
	
	
	\maketitle
	

	\section{Introduction}
	The aim of this paper is to study the uniqueness of global weak solutions to the initial/boundary-value problem
	\begin{equation}\label{eqt}
		\left\{
		\begin{aligned}
			&u_t = (uv  u_x)_x -  (u^2 v v_x)_x + uv,\;\;&x\in\Omega,\;t>0,\\
			&v_t =  v_{xx} - uv, \;\;&x\in\Omega,\;t>0,\\
			&u v u_x - u^2 v  v_x|_{\pa \Omega} = 0,\,v_x|_{\pa \Omega} = 0,&t>0,\\
			&u|_{t=0}=u^0,\,v|_{t=0}=v^0,&x\in\Omega,\\
		\end{aligned}
		\right .
	\end{equation}
	where $\Omega\subset\R$ is a bounded open interval, which describes the evolution of food-consuming bacterial populations in nutrient-poor environments. The equations have been proposed to describe the dynamics in collections of Bacillus subtilis (\cite{F.C.R}, \cite{RGP}).
	It has been shown by Winkler \cite{winkler_1} that \eqref{eqt} with assumption
	\begin{equation}\label{reg_init_time}
		\left\{
		\begin{aligned}
			&\mbox{~$u^0 \in C^{\vartheta}(\overline{\Omega})$ for some $\vartheta \in (0,1)$,~with~$u^0 \geq 0$ and~$\int_{\Omega} \ln u^0 dx > - \infty $, }\\
			&\mbox{~$v^0 \in W^{1,\infty}(\Omega)$ satisfies $v^0 > 0$ in~$\overline{\Omega}$ }
		\end{aligned}
		\right .
	\end{equation}
	admits a global weak solution $(u,v)$ in the sense that for any $\varphi \in C_0^\infty([0,\infty)\times \overline{\Omega})$,
	\begin{align}
		&-\int_0^\infty \int_{\Omega} u \varphi_t dxdt
		- \int_{\Omega} u^0 \varphi(\cdot,0) dx\label{eqt_u_weak}\\
		&\qquad\qquad
		= -\int_0^\infty \int_{\Omega}uv  u_x \varphi_x dxdt  +\int_0^\infty \int_{\Omega} u^2 v  v_x\varphi_x dxdt
		+ \int_0^\infty \int_{\Omega} uv \varphi dxdt,\nonumber\\
		&-\int_0^\infty \int_{\Omega} v \varphi_t dxdt
		- \int_{\Omega} v^0 \varphi(\cdot,0) dx
		= -\int_0^\infty \int_{\Omega} v_x\varphi_x dxdt
		- \int_0^\infty \int_{\Omega} uv \varphi dxdt,\label{eqt_v_weak}
	\end{align}
	which satisfies
	\begin{equation}\label{est}
		\left\{
		\begin{aligned}
			&0\leq u \in C^0([0,\infty)\times\overline{\Omega}) \cap L^\infty((0,\infty)\times \Omega ),~
			u_x \in L^2_{loc}([0,\infty)\times\overline{\Omega})  ,~\\
			&v \in C^0([0,\infty)\times\overline{\Omega})
			\cap C^{1,2}((0,\infty)\times\overline{\Omega})
			\cap L^\infty((0,\infty)\times \Omega ),~
			v_x \in L^\infty((0,\infty)\times \Omega ),~\\
			&\mbox{for any fixed}~T, ~v \geq c(T) > 0.
		\end{aligned}
		\right .
	\end{equation}
	The two-dimensional case was treated by Zhang and Li \cite{Z.L}.
	
	Our main result reads as follows.
	\begin{theorem}\label{thm_main}
		Let $(u,v)$ be a weak solution to \eqref{eqt}-\eqref{reg_init_time} satisfying \eqref{eqt_u_weak}-\eqref{est}. If
		$1/u^0 \in L^1(\Omega)$, then $(u,v)$ is unique.
	\end{theorem}
	
	Without loss of generality, we suppose that $\Omega = (0,1)$. The key idea of the proof is the use of $w(t,x) = \int_0^x (u(t,y) + v(t,y)) dy$, such that the uniqueness of $(u,v)$ is equivalent to that of $(w,v)$. More precisely, it follows from the $L^1(\Omega)$-conservation of $u + v$ that $w(t,0)\equiv0$ and $w(t,1)\equiv \|u^0+v^0\|_{L^1(\Omega)}$ which imply that $(w_1 - w_2)|_{\partial\Omega}=0$ in Lemma 4. The boundedness of~$1/u$ in~$L^\infty(0,\infty;L^1(\Omega))$(Lemma \ref{u^-1}) is used to control the term~$\frac{v_1^2(v_1 - v_2)^2}{u_1 + u_2}$ in \eqref{ineqt_v_pre} of Lemma \ref{lem_w^2+v^2}. We obtain Lemma \ref{lem_w^2+v^2} from Lemmas \ref{u^-1}-\ref{lem_w_boundary} and hence the uniqueness of $(w,v)$ by Gronwall's inequality.
	
	\section{Proof of Theorem \ref{thm_main}}
	\begin{lemma}\label{u^-1}
		If $1/u^0 \in L^1(\Omega)$, then $\Vert  u^{-1}  \Vert_{L^\infty(0,\infty;L^1(\Omega))} \leq C$.
	\end{lemma}
	\begin{proof}
		Applying a standard mollification procedure,
		we  prove that $t\mapsto \int_{\Omega} \frac{1}{u + \eps} dx$ is continuous on~$[0,T]$ and $\frac{d}{dt}\int_{\Omega} \frac{1}{u + \eps} dx = -\langle \pa_t u, \frac{1}{(u + \eps)^2} \rangle,$ where $\langle \cdot,\cdot \rangle$ is the dual product between $H^1(\Omega)^\prime$ and $H^1(\Omega)$.
		It follows from the weak formulation of \eqref{eqt} and Young inequality that
		\begin{align*}
			&\frac{d}{dt}\int_{\Omega} \frac{1}{u + \eps} dx
			=-2\int_{\Omega} uv  \frac{(u_x)^2}{(u + \eps)^3} dx
			+ 2 \int_{\Omega} u^2 v v_x \frac{u_x}{(u + \eps)^3} dx
			-\int_{\Omega}  uv \frac{1}{(u + \eps)^2} dx\\
			\leq&-2\int_{\Omega} uv  \frac{(u_x)^2}{(u + \eps)^3} dx
			+ \int_{\Omega} uv  \frac{(u_x)^2}{(u + \eps)^3} dx
			+ C \int_{\Omega} u^3 v (v_x)^2 \frac{1}{(u + \eps)^3} dx
			-\int_{\Omega}  uv \frac{1}{(u + \eps)^2} dx \\
			\leq&C \int_{\Omega} u^3 v (v_x)^2 \frac{1}{(u + \eps)^3} dx
			\leq C \int_{\Omega}  v (v_x)^2 dx.
		\end{align*}
		In view of \eqref{est}, one has
		\begin{align*}
			\int_{\Omega} \frac{1}{u + \eps} dx
			\leq C + \int_{\Omega} \frac{1}{u_0 + \eps} dx
			\leq C + \int_{\Omega} \frac{1}{u_0} dx
			\leq C,
		\end{align*}
		and hence the conclusion by Fatou's Lemma.
	\end{proof}
	
	Let $(u_i,v_i)\, (i=1,2)$ be two weak solutions to \eqref{eqt}-\eqref{reg_init_time} satisfying \eqref{eqt_u_weak}-\eqref{est}. Define that  $w_i(t,x) = \int_0^x u_i(t,y) + v_i(t,y) dy\, (i=1,2)$. Let $T > 0$ be any fixed constant.
	\begin{lemma}\label{eqt_w^2}
		For any $t \in (0,T)$, one has
		\begin{align*}
			\int_{\Omega} \frac{1}{2} (w_1 - &w_2)^2 dx
			=\frac{1}{2}\int_0^t\int_{\Omega}   \Big( v_1  (u_1^2)_x
			-  v_2  (u_2^2)_x \Big)  (w_1 - w_2) dxds\\
			&- \int_0^t\int_{\Omega}  (u_1^2 v_1 v_{1x} - u_2^2 v_2 v_{2x})(w_1 - w_2) dxds
			+ \int_0^t\int_{\Omega}   (v_{1x} - v_{2x}) (w_1 - w_2)dxds.
		\end{align*}
	\end{lemma}
	\begin{proof}
		In view of \eqref{eqt_u_weak}, for any $\psi \in C_0^\infty([0,T)\times \overline{\Omega})$, one has
		\begin{align}
			&-\int_0^T\int_{\Omega}  u_i   \pa_t \psi dxdt
			-\int_{\Omega}   u_i^0 \psi(0,\cdot) dx
			= -\int_0^T\int_{\Omega} f_i \psi_x  dxdt
			+ \int_0^T\int_{\Omega}  u_i v_i \psi  dxdt,~\nonumber
		\end{align}
		where $f_i = u_i v_i u_{ix} - u_i^2 v_i  v_{ix}.$
		Let $\tilde{ \Omega } = (-1,2)$, extend $u_i$ to~$\tilde{ \Omega }$ by mirroring on $x=0$ and $x=1$, that is
		\begin{align}
			\bar u_i(t,x) = \left\{
			\begin{aligned}
				&u_i(t,-x),~ -1 \leq x < 0,\\
				&u_i(t,x),~ 0 \leq x < 1,\\
				&u_i(t,2-x),~ 1 \leq x < 2.
			\end{aligned}
			\right .
		\end{align}
		Extend $u_i^0,v_i,v_i^0$ in a similar way and denote the extended function as $\bar u_i^0,\bar v_i,\bar v_i^0$. Denote $\bar f = \bar u_i \bar v_i \bar u_{ix} - \bar u_i^2 \bar v_i  \bar v_{ix}$. Then for any $\psi \in C_0^\infty([0,T)\times\tilde{ \Omega })$, one has
		\begin{align}
			&-\int_0^T\int_{\tilde{ \Omega }}  \bar u_i   \pa_t \psi dxdt
			-\int_{\tilde{ \Omega }}   \bar u_i^0 \psi(0,\cdot) dx
			= -\int_0^T\int_{\tilde{ \Omega }} \bar f_i \psi_x  dxdt
			+ \int_0^T\int_{\tilde{ \Omega }}  \bar u_i \bar v_i \psi  dxdt. \label{eqt_olu}
		\end{align}
		Denote $\rho^\eps$ as a standard mollifier with respect to space. Replacing $\psi$ in \eqref{eqt_olu} by $\rho^\eps* \psi$, it follows from the symmetry of $\rho^\eps$ and Fubini's theorem that
		\begin{align*}
			&-\int_0^T\int_{\tilde{ \Omega }}  \rho^\eps*\bar u_i    \pa_t \psi dxdt
			-\int_{\tilde{ \Omega }}  \rho^\eps* \bar u_i^0  \psi(0,\cdot) dx\\
			=& -\int_0^T\int_{\tilde{ \Omega }} \rho^\eps*\bar f_i  \psi_x  dxdt
			+ \int_0^T\int_{\tilde{ \Omega }}  \rho^\eps*(\bar u_i \bar v_i)  \psi  dxdt.
		\end{align*}
		Integrating by parts and replacing $x$ with $y$, one has
		\begin{align}
			&-\int_0^T\int_{\tilde{ \Omega }}  \rho^\eps*\bar u_i    \pa_t \psi dydt
			-\int_{\tilde{ \Omega }}  \rho^\eps* \bar u_i^0  \psi(0,\cdot) dy\nonumber\\
			=& \int_0^T\int_{\tilde{ \Omega }} (\rho^\eps*\bar f_i)_y  \psi  dydt
			+ \int_0^T\int_{\tilde{ \Omega }}  \rho^\eps*(\bar u_i \bar v_i)  \psi  dydt.\label{eqt_olu_2}
		\end{align}
		For the above $\psi \in C_0^\infty([0,T)\times \tilde{ \Omega })$, any fixed $x \in (0,1)$,~let~$\psi_\delta(t,y) = \rho^\delta * 1_{[0,x]}(y) \psi(t,x)$, where $1_{[0,x]}$ is the characteristic function on $[0,1]$. It holds (up to a subsequence) that
		\begin{align}
			\begin{aligned}
				&\psi_\delta \in C_0^\infty([0,T)\times \tilde{ \Omega }),~|\psi_\delta| \leq C,
				\rho^\delta * 1_{[0,x]} \to 1_{[0,x]} ~a.e..
			\end{aligned}
		\end{align}
		Replacing $\psi$ by $\psi_\delta$ in \eqref{eqt_olu_2} and letting $\delta \to 0$, it follows from Lebesgue's Dominated Convergence theorem that for any $\psi \in C_0^\infty([0,T)\times \tilde{ \Omega })$, $x \in (0,1)$,
		\begin{align*}
			&-\int_0^T\left(  \int_0^x  \rho^\eps*\bar u_i    dy  \right) \pa_t \psi(t,x) dt
			-\left(  \int_0^x  \rho^\eps* \bar u_i^0   dy  \right) \psi(0,x)\\
			&\qquad\qquad\qquad\qquad
			= \int_0^T\left(  \int_0^x (\rho^\eps*\bar f_i)_y  dy  \right) \psi(t,x) dt
			+ \int_0^T\left(  \int_0^x  \rho^\eps*(\bar u_i \bar v_i)  dy  \right)  \psi(t,x) dt.\nonumber
		\end{align*}
		Applying the Newton-Leibniz formula, one has
		\begin{align*}
			&-\int_0^T\left(  \int_0^x  \rho^\eps*\bar u_i    dy  \right) \pa_t \psi(t,x) dt
			-\left(  \int_0^x  \rho^\eps* \bar u_i^0   dy  \right) \psi(0,x)\\
			&
			= \int_0^T \rho^\eps*\bar f_i(x)  \psi(t,x) dt
			- \int_0^T \rho^\eps*\bar f_i(0)  \psi(t,x) dt
			+ \int_0^T\left(  \int_0^x  \rho^\eps*(\bar u_i \bar v_i)  dy   \right) \psi(t,x) dt.\nonumber
		\end{align*}
		Integrating on $(0,1)$ with respect to $x$, letting $\eps \to 0$, recalling the boundary conditions and the definition of $\bar f_i$, we have for any $\psi \in C_0^\infty([0,T)\times \tilde{ \Omega })$,
		\begin{align}
			&-\int_0^1\int_0^T\left(  \int_0^x  \bar u_i    dy  \right) \pa_t \psi dtdx
			-\int_0^1\int_0^x  \bar u_i^0   dy \psi(0,\cdot)dx\nonumber\\
			&\qquad\qquad\qquad\qquad\qquad\qquad\qquad\qquad
			= \int_0^1\int_0^T \bar f_i  \psi dtdx
			+ \int_0^1\int_0^T\left(  \int_0^x  \bar u_i \bar v_i  dy  \right)  \psi dtdx. \nonumber
		\end{align}
		Furthermore, the above equation holds for $\psi \in C_0^\infty([0,T)\times \overline{\Omega})$. By the definition of $\bar u_i,\bar u_i^0,\bar f_i,\bar v_i$, one has for any $\psi \in C_0^\infty([0,T)\times \overline{\Omega})$,
		\begin{align}\label{eqt_olu_7}
			&-\int_0^1\int_0^T\left(  \int_0^x  u_i    dy  \right) \pa_t \psi dtdx
			-\int_0^1\int_0^x u_i^0   dy \psi(0,\cdot)dx\nonumber\\
			&\qquad\qquad\qquad\qquad\qquad\qquad
			= \int_0^1\int_0^T f_i  \psi dtdx
			+ \int_0^1\int_0^T\left(  \int_0^x  u_i v_i  dy  \right)  \psi dtdx.
		\end{align}
		Similarly, we have that for any $\psi \in C_0^\infty([0,T)\times \overline{\Omega})$,
		\begin{align}\label{eqt_olv_7}
			&-\int_0^1\int_0^T\left(  \int_0^x  v_i    dy  \right) \pa_t \psi dtdx
			-\int_0^1\int_0^x v_i^0   dy \psi(0,\cdot)dx\nonumber\\
			&\qquad\qquad\qquad\qquad\qquad\qquad
			= \int_0^1\int_0^T g_i  \psi dtdx
			- \int_0^1\int_0^T\left(  \int_0^x  u_i v_i  dy  \right)  \psi dtdx,
		\end{align}
		where $g_i = v_{ix}$.
		Recalling that $w_i = \int_{0}^{x} (u_i + v_i) dy$ and~$u_1^0 = u_2^0,~v_1^0 = v_2^0$, canceling out the last terms in \eqref{eqt_olu_7} and \eqref{eqt_olv_7}, we have for any $\psi \in C_0^\infty([0,T)\times \overline{\Omega})$,
		\begin{align}\label{eqt_w_2}
			-\int_0^1\int_0^T (w_1 - w_2) \pa_t \psi dtdx
			= \int_0^1\int_0^T (f_1 - f_2)  \psi dtdx
			+ \int_0^1\int_0^T (g_1 - g_2)  \psi dtdx.
		\end{align}
		It follows from Fubini's theorem that
		\begin{align}\label{eqt_w_3}
			-\int_0^T\int_{\Omega} (w_1 - w_2) \pa_t \psi dxdt
			= \int_0^T\int_{\Omega} (f_1 - f_2)  \psi dxdt
			+ \int_0^T\int_{\Omega} (g_1 - g_2)  \psi dxdt.
		\end{align}
		Applying a standard mollification procedure, we prove that $t\mapsto \Vert  w_1 - w_2  \Vert_{L^2(\Omega)}^2$ is continuous on $[0,T]$ which together with the definition of $f_i,g_i$ implies that
		\begin{align*}
			&\frac{d}{dt}\int_{\Omega} \frac{1}{2} (w_1 - w_2)^2 dx
			= \int_{\Omega} (f_1 - f_2)(w_1 - w_2)   dx
			+ \int_{\Omega} (g_1 - g_2)(w_1 - w_2)   dx\\
			=&\frac{1}{2} \int_{\Omega}   \Big( v_1  (u_1^2)_x
			-  v_2  (u_2^2)_x \Big)  (w_1 - w_2) dx
			-  \int_{\Omega}  (u_1^2 v_1 v_{1x} - u_2^2 v_2 v_{2x})(w_1 - w_2) dx\\
			&+  \int_{\Omega}   (v_{1x} - v_{2x}) (w_1 - w_2)dx.
		\end{align*}
		Integrating over $(0,t)$, we finish the proof.
	\end{proof}

	\begin{lemma}\label{lem_w_boundary}
		For any $t \in (0,T)$, $(w_1 - w_2)|_{\pa \Omega} = 0$.
	\end{lemma}
	\begin{proof}
		Clearly  $(w_1 - w_2)|_{x=0}=0$.  We only need to prove $(w_1 - w_2)|_{x=1}=0$.
		It follows from \eqref{eqt_u_weak} and \eqref{eqt_v_weak} that for any $\varphi \in C^\infty([0,T]\times\overline{\Omega})$,
		\begin{align}
			&-\int_0^t\int_{\Omega} (u_i + v_i) \pa_t\varphi dxds
			+\int_{\Omega} (u_i + v_i) \varphi\Big|_{s=t} dx
			-\int_{\Omega} (u_i + v_i) \varphi\Big|_{s=0} dx\\
			&\qquad\qquad\qquad\qquad=
			-\int_0^t\int_{\Omega} u_i v_i u_{ix} \varphi_x dxds
			+ \int_0^t\int_{\Omega} u_i^2 v_i v_{ix} \varphi_x dxds
			-\int_0^t \int_{\Omega} v_{ix}\varphi_x dxdt . \nonumber
		\end{align}
		Letting $\varphi = 1$, one has
		$ \int_{\Omega} (u_i + v_i) \Big|_{s=t} dx=\int_{\Omega} (u_i + v_i) \Big|_{s=0} dx,$
		and hence the conclusion by noting that $w_i|_{x=1} = \int_{\Omega} (u_i + v_i)\Big|_{s=t} dx$ and $u_1^0 = u_2^0,v_1^0 = v_2^0$.
	\end{proof}

	\begin{lemma}\label{lem_w^2+v^2}
		There exists $C > 0$ such that for any $t \in (0,T)$,
		\begin{align}\label{ineqt_wv}
			&\int_{\Omega} \frac{1}{2} (w_1 - w_2)^2 dx
			+\int_{\Omega} \frac{1}{2}(v_1 - v_2)^2  dx
			+ \frac{c(T)}{32}\int_0^t\int_{\Omega} (u_1 + u_2) (w_{1x} - w_{2x})^2  dxds\\
			&\qquad
			+\frac{1}{8}\int_0^t\int_{\Omega} (v_{1x} - v_{2x})^2  dxds
			\leq
			C\int_0^t\int_{\Omega}(w_1 - w_2)^2 dxds
			+ C\int_0^t\int_{\Omega} (v_1 - v_2)^2  dxds \nonumber,
		\end{align}
		where $c(T)>0$ is the lower bound of $v_i$ in \eqref{est}.
	\end{lemma}
	\begin{proof}
		{\it Step $1$: estimate of $\Vert  w_1 - w_2  \Vert_{L^2(\Omega)}^2$.}
		Recalling Lemma \ref{eqt_w^2},
		\begin{align*}
			\int_{\Omega} \frac{1}{2} (w_1 - w_2)^2 dx
			=& \frac{1}{2}\int_0^t\int_{\Omega}   \Big( v_1  (u_1^2)_x
			-  v_2  (u_2^2)_x \Big)  (w_1 - w_2) dxds\\
			&- \int_0^t\int_{\Omega}  (u_1^2 v_1 v_{1x} - u_2^2 v_2 v_{2x})(w_1 - w_2) dxds\\
			&+ \int_0^t\int_{\Omega}   (v_{1x} - v_{2x}) (w_1 - w_2)dxds\\
			=&: I_1 + I_2 + I_3.
		\end{align*}
		
		We first estimate $I_1$.
		\begin{align*}
			I_1=& \frac{1}{2}\int_0^t\int_{\Omega}   v_1 (u_1^2
			-  u_2^2)_x  (w_1 - w_2) dxds
			+ \frac{1}{2}\int_0^t\int_{\Omega}   (v_1 - v_2)
			(u_2^2)_x  (w_1 - w_2) dxds.
		\end{align*}
		Integrating by parts, in view of Lemma \ref{lem_w_boundary}, one has
		\begin{align*}
			I_1=& -\frac{1}{2}\int_0^t\int_{\Omega}   v_{1x}  (u_1^2
			-  u_2^2)
			(w_1 - w_2) dxds
			-\frac{1}{2}\int_0^t\int_{\Omega}   v_1  (u_1^2
			-  u_2^2)
			(w_1 - w_2)_x dxds\\
			&- \frac{1}{2}\int_0^t\int_{\Omega}   (v_1 - v_2)_x
			u_2^2   (w_1 - w_2) dxds
			- \frac{1}{2}\int_0^t\int_{\Omega}   (v_1 - v_2)
			u_2^2   (w_1 - w_2)_x dxds.
		\end{align*}
		It follows from
		\begin{align}\label{eqt_u_1^2-u_2^2}
			u_1^2 -  u_2^2
			=(u_1 + u_2)  (u_1 - u_2)
			=&(u_1 + u_2) (w_{1x} - v_1 - w_{2x} + v_2)\\
			=&(u_1 + u_2) (w_{1x} - w_{2x})
			-(u_1 + u_2) (v_1 - v_2) \nonumber
		\end{align}
		that
		\begin{align*}
			I_1=&-\frac{1}{2}\int_0^t\int_{\Omega}   v_{1x}  (u_1 + u_2) (w_{1x} - w_{2x})
			(w_1 - w_2) dxds\\
			&+\frac{1}{2}\int_0^t\int_{\Omega}   v_{1x}  (u_1 + u_2) (v_1 - v_2)
			(w_1 - w_2) dxds\\
			&-\frac{1}{2}\int_0^t\int_{\Omega}   v_1  (u_1 + u_2) (w_{1x} - w_{2x})
			(w_1 - w_2)_x dxds\\
			&+\frac{1}{2}\int_0^t\int_{\Omega}   v_1  (u_1 + u_2) (v_1 - v_2)
			(w_1 - w_2)_x dxds\\
			&- \frac{1}{2}\int_0^t\int_{\Omega}   (v_1 - v_2)_x
			u_2^2   (w_1 - w_2) dxds
			- \frac{1}{2}\int_0^t\int_{\Omega}   (v_1 - v_2)
			u_2^2   (w_1 - w_2)_x dxds\\
			=:&I_{11} + \cdots + I_{16}.
		\end{align*}
		Recalling $v_i > c(T)> 0$, one has
		\begin{align*}
			I_{13}=& -\frac{1}{2}\int_0^t\int_{\Omega}   v_1  (u_1 + u_2) (w_{1x} - w_{2x})^2  dxds
			\leq - \frac{c(T)}{2}\int_0^t\int_{\Omega}    (u_1 + u_2)(w_{1x} - w_{2x})^2  dxds.
		\end{align*}
		It follows from Young's inequality and \eqref{est} that
		\begin{align*}
			I_{11}
			\leq& \frac{c(T)}{8}\int_0^t\int_{\Omega} (u_1 + u_2) (w_{1x} - w_{2x})^2  dxds
			+ C\int_0^t\int_{\Omega}  v_{1x}^2  (u_1 + u_2) (w_1 - w_2)^2 dxds\\
			\leq& \frac{c(T)}{8}\int_0^t\int_{\Omega} (u_1 + u_2) (w_{1x} - w_{2x})^2  dxds
			+ C\int_0^t\int_{\Omega}(w_1 - w_2)^2 dxds.
		\end{align*}
		Similarly
		\begin{align*}
			I_{12}
			\leq& C\int_0^t\int_{\Omega}  (v_1 - v_2)^2  dxds
			+ C\int_0^t\int_{\Omega} (w_1 - w_2)^2 dxds,~\\
			I_{14}
			\leq& \frac{c(T)}{8}\int_0^t\int_{\Omega} (u_1 + u_2) (w_{1x} - w_{2x})^2  dxds
			+ C\int_0^t\int_{\Omega}(v_1 - v_2)^2 dxds,~\\
			I_{15}
			\leq& \frac{1}{2}\int_0^t\int_{\Omega}   (v_{1x} - v_{2x})^2 dxds
			+ C \int_0^t\int_{\Omega}  (w_1 - w_2)^2 dxds,~\\
			I_{16}
			\leq&\frac{c(T)}{8}\int_0^t\int_{\Omega}
			(u_1 + u_2)  (w_{1x} - w_{2x}) ^2  dxds
			+ C\int_0^t\int_{\Omega}   (v_1 - v_2)^2 dxds.
		\end{align*}
		In conclusion, one has
		\begin{align*}
			I_1\leq&
			-\frac{c(T)}{8}\int_0^t\int_{\Omega}
			(u_1 + u_2)  (w_{1x} - w_{2x}) ^2  dxds
			+ \frac{1}{2}\int_0^t\int_{\Omega}   (v_{1x} - v_{2x})^2 dxds\\
			&\qquad\qquad\qquad + C\int_0^t\int_{\Omega}   (v_1 - v_2)^2 dxds
			+ C\int_0^t\int_{\Omega}   (w_1 - w_2)^2 dxds.
		\end{align*}
		
		Next we estimate $I_2$.
		\begin{align*}
			I_2
			= &-  \int_0^t\int_{\Omega}  (u_1^2 - u_2^2)v_1 v_{1x}(w_1 - w_2) dxds\\
			&-  \int_0^t\int_{\Omega}  (v_1  - v_2)v_{1x} u_2^2(w_1 - w_2) dxds
			-  \int_0^t\int_{\Omega}  (v_{1x} - v_{2x})v_2 u_2^2(w_1 - w_2) dxds,~
		\end{align*}
		In view of \eqref{eqt_u_1^2-u_2^2}, one has
		\begin{align*}
			I_2
			= &-  \int_0^t\int_{\Omega}  (u_1 + u_2) (w_{1x}- w_{2x}) v_1 v_{1x}(w_1 - w_2) dxds\\
			&+  \int_0^t\int_{\Omega}  (u_1 + u_2) (v_1 - v_2) v_1 v_{1x}(w_1 - w_2) dxds\\
			&-  \int_0^t\int_{\Omega}  (v_1  - v_2)v_{1x} u_2^2(w_1 - w_2) dxds
			-  \int_0^t\int_{\Omega}  (v_{1x} - v_{2x})v_2 u_2^2(w_1 - w_2) dxds.
		\end{align*}
		Similar to the estimating of $I_1$, we have
		\begin{align*}
			I_2
			\leq& \frac{c(T)}{16}\int_0^t\int_{\Omega} (u_1 + u_2) (w_{1x} - w_{2x})^2  dxds
			+\frac{1}{8}\int_0^t\int_{\Omega} (v_{1x} - v_{2x})^2  dxds\\
			&\qquad\qquad\qquad+ C\int_0^t\int_{\Omega}(w_1 - w_2)^2 dxds
			+ C\int_0^t\int_{\Omega} (v_1 - v_2)^2  dxds.
		\end{align*}
		The  estimate of $I_3$ follows from Young's inequality
		\begin{align*}
			I_3= \int_0^t\int_{\Omega}  (v_{1x} - \!v_{2x}) (w_1 - w_2)dxds
			\leq  \frac{1}{8}\int_0^t\int_{\Omega} (v_{1x} -\! v_{2x})^2  dxds
			+ C\int_0^t\int_{\Omega} (w_1 - \!w_2)^2  dxds .
		\end{align*}
		Combining the estimates of $I_1,I_2$ and $I_3$, one has
		\begin{align}
			\int_{\Omega} \frac{1}{2} (w_1 - w_2)^2 dx
			\leq&-\frac{c(T)}{16}\int_0^t\int_{\Omega} (u_1 + u_2) (w_{1x} - w_{2x})^2  dxds
			+\frac{3}{4}\int_0^t\int_{\Omega} (v_{1x} - v_{2x})^2  dxds\nonumber\\
			&\qquad+ C\int_0^t\int_{\Omega}(w_1 - w_2)^2 dxds
			+ C\int_0^t\int_{\Omega} (v_1 - v_2)^2  dxds. \label{ineqt_w}
		\end{align}
		
		{\it Step $2$: estimate of $\Vert  v_1 - v_2  \Vert_{L^2(\Omega)}^2$.} Recalling $u_1 - u_2 = w_{1x} - w_{2x} - (v_1 - v_2)$, one has
		\begin{align*}
			\pa_t (v_1 - v_2)
			&= (v_1 - v_2)_{xx} - (u_1v_1 - u_2v_2)
			= (v_{1x} - v_{2x})_{x} - (u_1 - u_2)v_1 - u_2(v_1 - v_2)\\
			&= (v_{1x} - v_{2x})_{x} - (w_{1x} - w_{2x})v_1
			+ (v_1 - v_2)v_1 - u_2(v_1 - v_2).
		\end{align*}
		Testing with $v_1 - v_2$ and integrating over $(0,t)\times\Omega$, we have
		\begin{align*}
			\int_{\Omega} \frac{1}{2}(v_1 - v_2)^2  dx
			= &-\int_0^t\int_{\Omega} (v_{1x} - v_{2x})^2 dxds
			- \int_0^t\int_{\Omega} (w_{1x} - w_{2x})v_1(v_1 - v_2)  dxds\\
			&+ \int_0^t\int_{\Omega} (v_1 - v_2)v_1(v_1 - v_2)  dxds
			- \int_0^t\int_{\Omega} u_2(v_1 - v_2)(v_1 - v_2)  dxds.
		\end{align*}
		It follows from Young's inequility and \eqref{est} that for any $\delta > 0$,
		\begin{align}
			\int_{\Omega} \frac{1}{2}(v_1 - v_2)^2  dx
			\leq &-\int_0^t\int_{\Omega} (v_{1x} - v_{2x})^2 dxds
			+ \frac{c(T)}{32} \int_0^t\int_{\Omega} (u_1 + u_2 + \delta) (w_{1x} - w_{2x})^2  dxds\nonumber\\
			&+ C\int_0^t\int_{\Omega} \frac{v_1^2(v_1 - v_2)^2}{u_1 + u_2 + \delta} dxds
			+ C \int_0^t\int_{\Omega} (v_1 - v_2)^2  dxds.\label{ineqt_v_pre}
		\end{align}
		In view of Lemma \ref{u^-1} and \eqref{est}, one has
		\begin{align*}
			\int_0^t\int_{\Omega} \frac{v_1^2(v_1 - v_2)^2}{u_1 + u_2 + \delta} dxds
			\leq&  C \int_0^t\Vert  (v_1 - v_2)^2  \Vert_{L^\infty(\Omega)} 	
			\left(  \int_{\Omega} \frac{1}{u_1 + u_2 + \delta} dx  \right)ds\\
			\leq& C\int_0^t \Vert  (v_1 - v_2)^2  \Vert_{L^\infty(\Omega)}ds.
		\end{align*}
		It follows from $W^{1,1}(\Omega) \hookrightarrow L^\infty(\Omega)$ that
		\begin{align*}
			\int_0^t  \Vert  (v_1 - v_2)^2  \Vert_{L^\infty(\Omega)} ds
			\leq& C \int_0^t\int_{\Omega} (v_1 - v_2)^2 dxds
			+ 2C \int_0^t\int_{\Omega} |(v_1 - v_2)(v_1 - v_2)_x| dxds\\
			\leq& C \int_0^t\int_{\Omega} (v_1 - v_2)^2 dxds
			+ \frac{1}{8} \int_0^t\int_{\Omega} (v_{1x} - v_{2x})^2 dxds.
		\end{align*}
		Therefore
		\begin{align}\label{ineqt_v_pre_2}
			\int_{\Omega} \frac{1}{2}(v_1 - v_2)^2  dx
			= &-\frac{7}{8}\int_0^t\int_{\Omega} (v_{1x} - v_{2x})^2 dxds\\
			+ \frac{c(T)}{32} &\int_0^t\int_{\Omega} (u_1 + u_2 + \delta) (w_{1x} - w_{2x})^2  dxds
			+ C \int_0^t\int_{\Omega} (v_1 - v_2)^2  dxds.\nonumber
		\end{align}
		Letting $\delta \to 0$ and in view of \eqref{ineqt_w}, we finishes the proof.
	\end{proof}
	\noindent {\it Proof of Theorem \ref{thm_main}.} It follows from Lemma \ref{lem_w^2+v^2} and Gronwall inequality that $w_1 = w_2,v_1=v_2$ and hence $u_1 = u_2$ for a.e.$(t,x) \in (0,T)\times \Omega $. The arbitrariness of $T$ concludes the proof.
	
	$\hfill \qedsymbol$


\begin{thebibliography}{11}
		\bibitem{F.C.R}J.~F. Leyva, C. M\'alaga and R.~G. Plaza, The effects of nutrient chemotaxis on bacterial aggregation patterns with non-linear degenerate cross diffusion, Phys. A {\bf 392} (2013), no.~22, 5644--5662; MR3102776
		
		\bibitem{winkler_1} M. Winkler, Does spatial homogeneity ultimately prevail in nutrient taxis systems? A paradigm for structure support by rapid diffusion decay in an autonomous parabolic flow, Trans. Amer. Math. Soc. {\bf 374} (2021), no.~1, 219--268; MR4188182
		
		\bibitem{RGP}R.~G. Plaza, Derivation of a bacterial nutrient-taxis system with doubly degenerate cross-diffusion as the parabolic limit of a velocity-jump process, J. Math. Biol. {\bf 78} (2019), no.~6, 1681--1711; MR3968978
		
		\bibitem{Z.L}Z. Zhang, and Y. Li, Boundedness in a two-dimensional doubly degenerate nutrient taxis system, (2024); arXiv:2405.20637
	\end{thebibliography}
\end{document}